\documentclass[12pt,leqno]{article}
\usepackage{hyperref}
\usepackage{soul}
\usepackage[doc]{optional}
\usepackage{tikz}
\usepackage{xcolor}
\definecolor{labelkey}{rgb}{0,0.08,0.45}
\definecolor{rekey}{rgb}{0,0.6,0.0}
\definecolor{Brown}{rgb}{0.45,0.0,0.05}
\usepackage{exscale,relsize}
\usepackage{amsmath}
\usepackage{amsfonts}
\usepackage{amssymb}
\usepackage{calc}
\usepackage{theorem}
\usepackage{pifont}      %needed by dingautolist
\usepackage{graphicx}
\newcommand{\scal}[2]{\langle{{#1},{#2}}\rangle}

\newcommand{\RR}{\ensuremath{\mathbb R}}

\newcommand{\RX}{\ensuremath{\,\left]-\infty,+\infty\right]}}
\newcommand{\RXX}{\ensuremath{\,\left[-\infty,+\infty\right]}}

\newcommand{\NN}{\ensuremath{\mathbb N}}

\newcommand{\menge}[2]{\big\{{#1} \mid {#2}\big\}}

\newcommand{\To}{\ensuremath{\rightrightarrows}}

\newcommand{\di}{\ensuremath{\operatorname{dist}}}

\newcommand{\dom}{\ensuremath{\operatorname{dom}}}

\newcommand{\gra}{\ensuremath{\operatorname{gra}}}

\newcommand{\inte}{\ensuremath{\operatorname{int}}}

\newcommand{\ran}{\ensuremath{\operatorname{ran}}}

\newcommand{\conv}{\ensuremath{\operatorname{conv}}}

\renewcommand{\phi}{\ensuremath{\varphi}}

\newtheorem{theorem}{Theorem}[section]

\newtheorem{fact}[theorem]{Fact}
\newtheorem{corollary}[theorem]{Corollary}
\newtheorem{proposition}[theorem]{Proposition}
\newtheorem{definition}[theorem]{Definition}

\theoremstyle{plain}{\theorembodyfont{\rmfamily}
}
\theoremstyle{plain}{\theorembodyfont{\rmfamily}
}
\theoremstyle{plain}{\theorembodyfont{\rmfamily}
}
\theoremstyle{plain}{\theorembodyfont{\rmfamily}
}
\theoremstyle{plain}{\theorembodyfont{\rmfamily}
\newtheorem{remark}[theorem]{Remark}}

\theoremstyle{plain}{\theorembodyfont{\rmfamily}
}

\begin{document}

%\sffamily

\title{\sffamily{Some results on the  convexity of the closure of the domain of a maximally monotone operator
 }}

\author{
 Jonathan M.
Borwein\thanks{CARMA, University of Newcastle, Newcastle, New South
Wales 2308, Australia. E-mail:
\texttt{jonathan.borwein@newcastle.edu.au}.  Distinguished Professor
King Abdulaziz University, Jeddah. }\;
and Liangjin\
Yao\thanks{CARMA, University of Newcastle,
 Newcastle, New South Wales 2308, Australia.
E-mail:  \texttt{liangjin.yao@newcastle.edu.au}.}}
 \vskip 3mm

\date{May 20, 2012}
\maketitle

\begin{abstract} \noindent
We provide a concise analysis about what is known regarding  when the closure of the domain of a maximally monotone operator on an arbitrary real Banach space  is convex.
In doing so, we also provide an affirmative answer to a problem posed by Simons.
\end{abstract}



\section{Introduction}

As discussed in \cite{Bor1,Bor2,Bor3,BorVan}, the two most central open questions in monotone operator theory are probably in a general real Banach space whether (i) the sum of two maximally monotone operators $A+B$ is maximally monotone under \emph{Rockafellar's constraint qualification} $\dom A \cap \inte \dom B\neq \emptyset$ \cite{Rock70}; and
whether (ii) $ \overline{\dom A}$ is always convex.
Rockafellar showed that $\overline{\dom A}$ is convex if $\inte\dom A\neq\varnothing$ \cite{Rock69}.

Since a positive answer to various restricted versions of (i) implies a positive answer to (ii) \cite{BorVan,Si2}, we have determined to make what progress we may with the later question.

\subsection{Preliminary matters}

Throughout this note, we assume that
$X$ is a general real Banach space with norm $\|\cdot\|$,
that $X^*$ is its continuous dual space, and that $\scal{\cdot}{\cdot}$ denotes
the usual pairing between these spaces.
Let $A\colon X\To X^*$
be a \emph{set-valued operator} (also known as multifunction)
from $X$ to $X^*$, i.e., for every $x\in X$, $Ax\subseteq X^*$,
and let
$\gra A = \menge{(x,x^*)\in X\times X^*}{x^*\in Ax}$ be
the \emph{graph} of $A$.
Then $A$ is said to be \emph{monotone} if
\begin{equation}
\big(\forall (x,x^*)\in \gra A\big)\big(\forall (y,y^*)\in\gra
A\big) \quad \scal{x-y}{x^*-y^*}\geq 0,
\end{equation}
and \emph{maximally monotone} if
no proper enlargement
(in the sense of graph inclusion) of $A$ is monotone.
 We say $(x,x^*)$ is \emph{monotonically related to}
$\gra A$ if
\begin{align*}
\langle x-y,x^*-y^*\rangle\geq0,\quad \forall (y,y^*)\in\gra
A.\end{align*}
We say $A$ is a \emph{linear relation} if $\gra A$ is a linear subspace.
Let $A:X\rightrightarrows X^*$ be such that $\gra A\neq\varnothing$. The
\emph{Fitzpatrick function} associated with $A$ \cite{Fitz88} is defined by
\begin{equation*}
F_A\colon X\times X^*\to\RX\colon
(x,x^*)\mapsto \sup_{(a,a^*)\in\gra A}
\big(\scal{x}{a^*}+\scal{a}{x^*}-\scal{a}{a^*}\big).
\end{equation*}
Monotone operators have proven to be a key class of objects
in modern Optimization and Analysis; see, e.g.,
the books
\cite{BC2011,BorVan,BurIus,ph,Si,Si2,RockWets,Zalinescu,Zeidlerlin,Zeidler} and the references therein.

Throughout, we adopt standard convex analysis notation.
Given a subset $C$ of $X$,
$\inte C$ is the \emph{interior} of $C$,
$\overline{C}$ is   the
\emph{norm closure} of $C$, and $\conv{C}$ is   the
\emph{convex hull} of $C$.
The \emph{indicator function} of $C$, written as $\iota_C$, is defined
at $x\in X$ by
\begin{align}
\iota_C (x)=\begin{cases}0,\,&\text{if $x\in C$;}\\
+\infty,\,&\text{otherwise}.
\end{cases}\end{align}
The \emph{distance function}  to   the set $C$, written as $\di (\cdot, C)$,
is defined by
$x\mapsto \inf_{c\in C}\|x-c\|$.
 For every $x\in
X$, the \emph{normal cone} operator of $C$ at $x$ is defined by
$N_C(x)= \menge{x^*\in X^*}{\sup_{c\in C}\scal{c-x}{x^*}\leq 0}$, if
$x\in C$; and $N_C(x):=\varnothing$, if $x\notin C$.

We also say a set $C$ is \emph{nearly convex} if $\overline{C}$ is convex.

Let $f\colon X\to \RX$. Then
$\dom f= f^{-1}(\RR)$ is the \emph{domain} of $f$, and
$f^*\colon X^*\to\RXX\colon x^*\mapsto
\sup_{x\in X}(\scal{x}{x^*}-f(x))$ is
the \emph{Fenchel conjugate} of $f$.
%The \emph{epigraph} of $f$ is $\epi f = \menge{(x,r)\in
%X\times\RR}{f(x)\leq r}$.
%The  \emph{lower semicontinuous hull} of $f$, denoted by $\overline{f}$,
%is the function defined at every $x\in X$ by
%$\overline{f}(x) = \inf\menge{t\in\RR}{(x,t)\in\overline{\epi f}}$.
%The  \emph{convex  hull} of $f$, denoted by $\conv{f}$,
%is  defined at every $x\in X$ by
%$(\conv{f})(x) = \inf\menge{t\in\RR}{(x,t)\in\conv\left[{\epi f}\right]}$.
We say $f$ is proper if $\dom f\neq\varnothing$.
Let $f$ be proper. The \emph{subdifferential} of
$f$ is defined by
   $$\partial f\colon X\To X^*\colon
   x\mapsto \{x^*\in X^*\mid(\forall y\in
X)\; \scal{y-x}{x^*} + f(x)\leq f(y)\}.$$
Let $Y$ be another real Banach space. We denote  $P_X: X\times Y\rightarrow
X\colon (x,y)\mapsto x$ and  $\NN:=\{1,2,3,\ldots\}$.

The paper is organized as follows.
In Section~\ref{s:aux}, we collect auxiliary results for future
reference and for the reader's convenience.

In Section~\ref{s:main}, we first present a sufficient condition for the closure of the domain of a maximally monotone operator $A$ to be convex in Theorem~\ref{coroll:nc2}.  To the best of our knowledge, this recaptures  and extends all known general conditions for  $\dom A$ to be nearly convex.
We then turn to our second
main result
(Theorem~\ref{ProdLm:3}) provides an affirmative answer to a problem
 posed by  Simons in
 \cite[Problem~28.3, page~112]{Si2}.

\begin{quotation}
\noindent
\emph{Let $A:X\rightrightarrows X^*$ be maximally monotone.
 Is it necessarily
true that
\begin{align*}
\overline{P_X\left[\dom F_A\right]}=\overline{\conv\left[\dom A\right]}\, ?
\end{align*}
}
 \end{quotation}
 Simons\cite{Si05}, showed the answer is yes under the constraint that
$\inte\dom A\neq\varnothing$. (This is also an immediate consequence of Corollary~\ref{cor:near} below.)

\section{Auxiliary Results}
\label{s:aux}

The basic building blocks we appeal to are as follows. We first record Fitzpatrick's result connecting $A$ to the corresponding Fitzpatrick function.

\begin{fact}[Fitzpatrick]
\emph{(See {\cite[Proposition~3.2, Theorem~3.4 and Corollary~3.9]{Fitz88}}.)}
\label{f:Fitz}
Let $A\colon X\To X^*$ be monotone with $\dom A\neq\varnothing$.
Then $F_A$ is proper lower semicontinuous, convex and $F_A=\langle\cdot,\cdot\rangle$ on $\gra A$. Moreover, if $A$ is maximally monotone,
for every $(x,x^*)\in X\times X^*$, the inequality
$$\scal{x}{x^*}\leq F_A(x,x^*)$$ is true,
and equality holds if and only if $(x,x^*)\in\gra A$.
\end{fact}

We turn to  two technical results of Simons that we shall exploit.

\begin{fact}[Simons]
\emph{(See \cite[Lemma~19.18]{Si2}.)}
\label{Siimu:1}
Let $A:X\To X^*$ be monotone with $\gra A\neq\varnothing$ and $(z,z^*)\in\dom F_A$.
 Let $F:\gra A\rightarrow\RR$ be such that $\inf F>0$. Then
there exists $r\in\RR$ such that
\begin{align*}
\frac{\langle z-a,z^*-a^*\rangle}{F(a,a^*)}\geq r,\quad \forall (a,a^*)\in\gra A.
\end{align*}
\end{fact}

\begin{fact}[Simons]
\emph{(See \cite[Thereom~27.5]{Si2}.)}
\label{Sidomain:1}
Let $A:X\To X^*$ be monotone. Assume the implication that
\begin{align*}
z\notin \overline{\dom A}\quad\Rightarrow\quad
\sup_{(a,a^*)\in\gra A}
\frac{\langle z-a, a^*\rangle}{\| z-a\|}=+\infty\end{align*}
holds. Then
\begin{align*}\overline{\dom A}=\overline{\conv\left[\dom A\right]}=\overline{P_X\left[\dom F_A\right]},\end{align*}
and hence $\dom A$ is nearly convex: $\overline{\dom A}$ is convex, since $P_X\left[\dom F_A\right]$ is the projection of a convex set.
\end{fact}
Let us recall two definitions:

 \begin{definition}[(FPV) and (BR)]\label{def0}
 Let $A:X\To X^*$ be maximally monotone.  We say $A$ is
\emph{type \emph{Fitzpatrick-Phelps-Veronas
}(FPV)} if  for every open convex set $U\subseteq X$ such that
$U\cap \dom A\neq\varnothing$, the implication
\begin{equation*}
x\in U\,\text{and}\,(x,x^*)\,\text{is monotonically related to $\gra
A\cap (U\times X^*)$} \Rightarrow (x,x^*)\in\gra A
\end{equation*}
holds.\\
We say $A$ is of \emph{of  ``Br{\o}nsted-Rockafellar" (BR) type} \cite{Si6}
if whenever $(x,x^*)\in X\times X^*$, $\alpha,\beta>0$ while
\begin{align*}\inf_{(a,a^*)\in\gra A} \langle x-a,x^*-a^*\rangle
>-\alpha\beta\end{align*} then there exists $(b,b^*)\in\gra A$ such
that $\|x-b\|<\alpha,\|x^*-b^*\|<\beta$.
 \end{definition}

Note that if $\le$ is used throughout we get a nominally slightly stronger property \cite{VV4}. We also remark that unlike most properties studied for monotone operators (BR) is an isometric property not a priori preserved under Banach space isomorphism \cite{BBWY}.

 Likewise  in \cite{VV1,VV4} a notion of \emph{Verona regularity}
 of a maximally monotone operator $A$ is introduced and studied.
In \cite[Theorems 1\&2]{VV6} it is shown that $A$ is Verona regular,
 if and only if, $A+M$ is maximally monotone
 whenever $M$ is maximally monotone and has bounded range, if and only if,
   $A+ \lambda\partial \| \cdot -x\|$ is maximally monotone for all $\lambda>0, x \in X$.

We finish this section with a result on linear mappings.

\begin{fact}\emph{(See \cite[Corollary~3.3]{Yao2}.)}\label{domain:L1}
Let $A:X\To X^*$ be  a maximally monotone linear relation, and $f:X\rightarrow\RX$ be a
proper lower semicontinuous convex function with
 $\dom  A\cap\inte\dom\partial f\neq\varnothing$.
Then $A+\partial f$ is of type  $(FPV)$. In particular $A$ is of type $(FPV)$.
\end{fact}

\section{Main results}
\label{s:main}

A useful general criterion for \eqref{conv} to hold is the following result abstracted from \cite{VV4}.
Let us denote $J_p := \frac 1p \partial \| \cdot \|^p.$ Thus, $J_2=J$ is the classical duality map, while $J_1$ is a maximally monotone operator with  bounded  range.

\begin{proposition}[Domain near convexity]\label{prop:nc}
Let $ 1 \leq p < \infty$ be given. Suppose that  $A$ is  monotone and that for all $z\notin\overline{\dom A}$ and all sufficiently large $\lambda >0$
 \begin{align}\label{cq:mm} (z,0)\,~ \mbox{\rm  is not monotonically related to} \gra \big(A+  \lambda J_p(\cdot -z)\big).\end{align}
 Then
   \begin{align}\label{conv} \overline{\dom A}=\overline{\conv\left[\dom A\right]}=\overline{P_X\left[\dom F_A\right]},\end{align}
   and in particular $\dom A$ is nearly convex.
   \end{proposition}

\begin{proof} Let $z\in X$ be such that
\begin{align}z\notin\overline{\dom A}\label{FPCGG:1}.\end{align}
\allowdisplaybreaks
We have $\alpha:= \di(z, \overline{\dom A})>0$.
Since $(z,0)$  is not monotonically related to $\gra\big(A +\lambda J_p(-z+\cdot)\big)$, there exist
$(a_{\lambda}, a_{\lambda}^*)\in\gra A$ and $b^*_{\lambda}\in  J_p(-z+a_{\lambda})$ such that
\begin{align*}
&\langle
 z-a_{\lambda}, {a^*_{\lambda}}+\lambda b^*_{\lambda}\rangle>0\\
&\Rightarrow \langle
 z-a_{\lambda}, {a^*_{\lambda}}\rangle>\lambda\langle -z+a_{\lambda}, b^*_{\lambda}\rangle\\
&\Rightarrow \langle
 z-a_{\lambda}, {a^*_{\lambda}}\rangle>\lambda\| -z+a_{\lambda}\|^p
 \quad \text{(by $b^*_{\lambda}\in J_p(-z+a_{\lambda})$)}\\
&\Rightarrow
\frac{ \langle z-a_{\lambda}, a^*_{\lambda}\rangle}{\| -z+a_{\lambda}\|}>
\lambda\| -z+a_{\lambda}\|^{p-1}\quad \text{(since $\| -z+a_{\lambda}\|\geq\alpha>0$)}\\
&\Rightarrow
\frac{\langle z-a_{\lambda}, a^*_{\lambda}\rangle}{\| -z+a_{\lambda}\|}
> \lambda\alpha^{p-1}\quad \text{(by $\| -z+a_{\lambda}\|\geq\alpha$)}\\
&\Rightarrow \sup_{\lambda}
\frac{\langle z-a_{\lambda}, a^*_{\lambda}\rangle}{\| z-a_{\lambda}\|}=+\infty.
\end{align*}

 By Fact~\ref{Sidomain:1}, $\overline{\dom A}=\overline{\conv\left[\dom A\right]}=\overline{P_X\left[\dom F_A\right]}$.
Hence  $\overline{\dom A}$ is convex.
\end{proof}

We now provide a potentially weaker criterion for  equality of $\overline{\conv\left[\dom A\right]}=\overline{P_X\left[\dom F_A\right]}$ for a monotone operator.

\begin{proposition}
\label{SiCodomain:1}
Let $A:X\To X^*$ be monotone. Assume the implication that
\begin{align}
z\notin\overline{\conv\left[\dom A\right]}\quad\Rightarrow\quad
\sup_{(a,a^*)\in\gra A}
\frac{\langle z-a, a^*\rangle}{\| z-a\|}=+\infty \label{Codm:ca1}\end{align}
holds. Then
\begin{align*}\overline{\conv\left[\dom A\right]}=\overline{P_X\left[\dom F_A\right]}.\end{align*}
\end{proposition}

\begin{proof}
By Fact~\ref{f:Fitz}, $\conv\left[\dom A\right]\subseteq P_X\left[\dom F_A\right]$
 and thus $\overline{\conv\left[\dom A\right]}\subseteq\overline{P_X\left[\dom F_A\right]}$.
It suffices to show that
\begin{align}
P_X\left[\dom F_A\right]\subseteq\overline{\conv\left[\dom A\right]}. \label{Codm:c1}
\end{align}
Suppose to the contrary that there exists $(z,z^*)\in\dom F_A$
such that $z\notin \overline{\conv\left[\dom A\right]}$.
Define $F:\gra A: (a,a^*)\mapsto\|z-a\|$. We have $\inf F=\di(z, \dom A)>0$.
Then by Fact~\ref{Siimu:1}, there exists $r\in\RR$ such that
for every $(a,a^*)\in\gra A$,
\begin{align*}
\|z^*\|-\frac{\langle z-a,a^*\rangle}{\|z-a\|}\geq\frac{\langle z-a,z^*\rangle}{\|z-a\|}-\frac{\langle z-a,a^*\rangle}{\|z-a\|}=\frac{\langle z-a,z^*-a^*\rangle}{\|z-a\|}\geq r.
\end{align*}
Then we have
\begin{align*}
\sup_{(a,a^*)\in\gra A}\frac{\langle z-a,a^*\rangle}{\|z-a\|}\leq \|z^*\|- r<+\infty, ~\text{which contradicts
\eqref{Codm:ca1}}.
\end{align*}
Hence \eqref{Codm:c1} holds.
\end{proof}

The proof of Proposition~\ref{SiCodomain:1} above   follows  closely that of \cite[Thereom~27.5]{Si2}.
We are now able to establish our first core result.

\begin{theorem}\label{coroll:nc2}
Let $ 1 \leq p < \infty$ be given and $A: X\rightrightarrows X^*$ be  monotone.
\begin{enumerate}\item \label{part:i}
First suppose that for all $z\notin\overline{\dom A}$ and all sufficiently large $\lambda >0$
 \begin{align*} A+  \lambda J_p(\cdot -z)\,~ \mbox{\rm  is maximally monotone}.\end{align*}
 Then $\dom A$ is nearly convex since
   \begin{align*}\overline{\dom A}=\overline{\conv\left[\dom A\right]}=\overline{P_X\left[\dom F_A\right]}.\end{align*}
   \item \label{part:ii}If we
  only suppose  that for all $z\notin\overline{\conv\left[\dom A\right]}$ and all sufficiently large $\lambda >0$
 \begin{align*} (z,0)\,~ \mbox{\rm  is not monotonically related to} \gra \big(A+  \lambda J_p(\cdot -z)\big).\end{align*}
 Then
\begin{align*}\overline{\conv\left[\dom A\right]}=\overline{P_X\left[\dom F_A\right]}.
\end{align*}
   \end{enumerate}
   \end{theorem}
\begin{proof}\emph{Part \ref{part:i}.}
Suppose $z\notin\overline{\dom A}$. Then $(z,0)\notin\gra\big(A+  \lambda J_p(\cdot -z)\big)$.
Since $A+  \lambda J_p(\cdot -z)$ is maximally monotone, $(z,0)$ cannot be monotonically related to $\gra \big(A+  \lambda J_p(\cdot -z)\big)$. Then we may directly apply Proposition~\ref{prop:nc}.

\noindent \emph{Part \ref{part:ii}.}
Suppose now that $z\notin\overline{\conv\left[\dom A\right]}$.
 We follow the lines of the proof of Proposition~\ref{prop:nc} to obtain
 $\sup_{(a,a^*)\in\gra A}
\frac{\langle z-a, a^*\rangle}{\| z-a\|}=+\infty$. We then apply Proposition~\ref{SiCodomain:1} and obtain
$$\overline{\conv\left[\dom A\right]}=\overline{P_X\left[\dom F_A\right]},$$
as asserted.
\end{proof}

When $A$ is maximally monotone,
Part (i) of Theorem~\ref{coroll:nc2} holds in lots of attractive cases.

\begin{corollary}[Conditions for near convexity]\label{cor:near} Let $A:X\rightrightarrows X^*$ be maximally monotone. Then $\overline{\dom A}=\overline{\conv\left[\dom A\right]}=\overline{P_X\left[\dom F_A\right]}$ (and thus $\dom A$ is nearly convex) as soon as any of the following hold:
\begin{enumerate}
\item\label{part:ia} $X$ is reflexive;
\item\label{part:ib} $A$ is Verona regular;
\item\label{part:ic}  $A$ is of type (FPV); as holds if
\item\label{part:id} $\inte \dom A \neq \emptyset$; or if
\item\label{part:ie} $\gra A$ is affine.

\end{enumerate}
\end{corollary}

\begin{proof}
 \emph{Part \ref{part:ia}.}
First,  in reflexive spaces, we have a Rockafellar's sum theorem implying that $A+\lambda J_p(\cdot-z)$ is always maximally monotone for every $z\in X$ since $\dom J_p(\cdot-z)=X$, see \cite{Rock70,Si2, BorVan}).

\noindent \emph{Part \ref{part:ib}.} Next observe that setting $p=1$ shows that \emph{Verona regular mappings} imply $\overline{\dom A}=\overline{\conv\left[\dom A\right]}=\overline{P_X\left[\dom F_A\right]}$, since as noted above \ the Veronas  show the maximal monotonicity of $A+\lambda J_1(\cdot-z)$ for all $\lambda>0,z \in X$ is equivalent to their definition.

\noindent \emph{Part \ref{part:ic}.}   Likewise, \cite[Corollary~3.6]{Yao3} gives another proof that type (FPV) mappings have this property. Or see \cite[Theorem~44.2]{Si2}, \cite[Proposition~5.2.1, page~107]{YaoPhD}.

 \noindent \emph{Part \ref{part:id}.} This holds for various reasons. For instance it is shown in  \cite{Bor2} that $A$ is type (FPV). It also follows from the sum theorem in \cite[Theorem~9.6.1]{Bor2} which applies to two maximally monotone operators $A,B$ such that $\inte \dom A \cap \inte \dom B \neq \emptyset,$ and so shows that the maximal monotonicity of $A+\lambda J_p(\cdot-z)$ for all $p \ge 1, \lambda >0,z \in X$.

  \noindent \emph{Part \ref{part:ie}.} By Fact~\ref{domain:L1}, we have that $A$ is type (FPV) and that
  $A+\lambda J_p(\cdot-z)$ is maximally monotone for all $\lambda>0,z \in X$. Then we directly apply
  Theorem~\ref{coroll:nc2}. Or see \cite[Theorem~46.1]{Si2}.
\end{proof}

In both  cases \ref{part:ib} and \ref{part:ic}, it is still entirely possible that all maximally monotone operators have the desired property. The only other general conditions we know for near convexity are a relativization of Part \ref{part:id}, namely,  $^{ic}(\conv\dom A)\neq\varnothing$  \cite[Corollary~5]{ZalVoi} and the result below which follows directly from the definition of operators of type (BR):

\begin{proposition}\label{BRFa:1}
Let $A:X\rightrightarrows X^*$ be maximally monotone
 and $(x,x^*)\in X\times X^*$. Assume that $A$
is of type (BR) and that $\inf_{(a,a^*)\in\gra A}\langle
x-a,x^*-a^*\rangle>-\infty$. Then $x\in\overline{\dom A}$ and
$x^*\in\overline{\ran A}$. In particular, \[\overline{\dom A}=\overline{P_X\left[\dom F_A\right]}\quad \mbox{and}
 \quad \overline{\ran A}=\overline{P_{X^*}\left[\dom F_A\right]}.\]
 Consequently, both $\dom A$ and $\ran A$ are nearly convex.
\end{proposition}

\begin{proof}
By an easy argument, see \cite[Lemma~1(a)\&(b)]{VV5} or \cite[Fact~2.10]{BBWY}, we have
$x\in\overline{\dom A}$ and
$x^*\in\overline{\ran A}$.

We now show that $\overline{\dom A}=\overline{P_X\left[\dom F_A\right]}$.
By Fact~\ref{f:Fitz}, it suffices to show that
\begin{align}
P_X\left[\dom F_A\right]\subseteq\overline{\dom A}.
\label{Brcp:1}
\end{align}
Let $(z,z^*)\in \dom F_A$. Then we have  $\inf_{(a,a^*)\in\gra A}\langle
x-a,x^*-a^*\rangle>-\infty$. Hence $z\in \overline{\dom A}$ by the above result. Thus
\eqref{Brcp:1} holds and then $\overline{\dom A}=\overline{P_X\left[\dom F_A\right]}$.
Similarly, we have $\overline{\ran A}=\overline{P_{X^*}\left[\dom F_A\right]}$.
Hence
$\dom A$ and $\ran A$ both are nearly convex.
\end{proof}

Note the different flavour of Proposition \ref{BRFa:1} with its implications for near convexity of the range and domain.
We recall also that all dense type (D) operators are of type (BR) (see \cite[Theorem~1.4(4)]{MarSva2}). A detailed analysis of (BR) and type (D) operators is made in \cite{BBWY}.
We now  come to our second main result.

\begin{theorem}\label{ProdLm:3}
Let $A:X\rightrightarrows X^*$ be maximally monotone.
Then
\begin{align*}
\overline{\conv\left[\dom A\right]}=\overline{P_X\left[\dom F_A\right]}.
\end{align*}
\end{theorem}

\begin{proof}
By Fact~\ref{f:Fitz}, it suffices to show that
\begin{align}
P_X\left[\dom F_A\right]\subseteq \overline{\conv\left[\dom A\right]}.
\label{Lsee:10}
\end{align}
Let $(z,z^*)\in\dom F_A$.  We shall show that
\begin{align}
z\in\overline{\conv\left[\dom A\right]}.
\label{Lsee:11}
\end{align}
Suppose to the contrary that
\begin{align}
z\notin\overline{\conv\left[\dom A\right]}.
\label{Lsee:12}
\end{align}
Define $B:X\rightrightarrows X^*$ by
\begin{align}
\gra B:=\gra A-\{(0,z^*)\}.
\end{align}
Then we have
\begin{align}
F_B(z,0)&=\sup_{(x,x^*)\in\gra B}\big\{\langle z,x^*\rangle-\langle x,x^*\rangle\big\}\nonumber\\
&=\sup_{(y,y^*)\in\gra A}\big\{\langle z,y^*-z^*\rangle-\langle y,y^*-z^*\rangle\big\}\nonumber\\
&=\sup_{(y,y^*)\in\gra A}\big\{\langle z,-z^*\rangle+\langle z,y^*\rangle+\langle y,z^*\rangle-\langle y,y^*\rangle\big\}\nonumber\\
&=\langle z,-z^*\rangle+F_A(z,z^*).\label{Lsee:13}
\end{align}
Since $(z,z^*)\in\dom F_A$, by \eqref{Lsee:13}, there exists $r\in\RR$ such that
\begin{align}
F_B(z,0)\leq r.\label{Lsee:14a}
\end{align}
By construction, $B$ is maximally monotone and $\dom B=\dom A$. Then
$z\notin\overline{\conv\left[\dom B\right]}$ by
\eqref{Lsee:12}. By the Separation theorem, there exist $\delta>0$ and  $y_0^*\in X^*$ with $\|y^*_0\|=1$
such that
\begin{align}
\langle y_0^*, z-b\rangle>\delta, \quad \forall b\in \conv\left[\dom B\right].
\label{Lsee:14}
\end{align}
Let $n\in\NN$.
Since $z\notin\overline{\conv\left[\dom B\right]}$, $(z, ny_0^*)\notin\gra B$.
By the maximal monotonicity of $B$, there exists $(b_n, b_n^*)\in\gra B$ such that
\begin{align}
&\langle z-b_n, b^*_n\rangle+\langle ny_0^*, b_n\rangle>\langle n y_0^*, z\rangle
\quad \Rightarrow\quad \langle z-b_n, b^*_n\rangle>\langle ny_0^*, z-b_n\rangle\nonumber\\
&\quad \Rightarrow \langle z-b_n, b^*_n\rangle>n\delta\quad\text{(by \eqref{Lsee:14})}.
\label{Lsee:15}
\end{align}
Then we have
\begin{align*}
F_B(z,0)&\geq \sup_{n\in\NN}\big\{\langle z-b_n, b^*_n\rangle\big\}\\
&\geq \sup_{n\in\NN}\big\{n\delta \big\}\\
&=+\infty,\end{align*}
which contradicts \eqref{Lsee:14a}.
Hence $z\in \overline{\conv\left[\dom A\right]}$ and in consequence \eqref{Lsee:10} holds.
\end{proof}

\begin{remark} As already noted
Theorem~\ref{ProdLm:3}
provides an affirmative answer to a question
 posed by  Simons in
 \cite[Problem~28.3, page~112]{Si2}.
\end{remark}

\textbf{Question:} The following question is still unresolved:
 \vskip 3mm

\begin{quotation}
\noindent Let $A:X\rightrightarrows X^*$ be maximally monotone.
Is $\overline{\dom A}$ necessarily convex?
In light of Theorem  ~\ref{ProdLm:3} this is equivalent to asking if
$$\overline{\dom A}=\overline{P_X\left[\dom F_A\right]}$$
always holds?
\end{quotation}

\section*{Acknowledgment}
Jonathan Borwein and Liangjin Yao
were both partially supported by the Australian Research Council.

%\small

\end{document}